\documentclass[11pt,reqno,tbtags]{amsart}

\usepackage{amsthm}
\usepackage{amssymb}
\usepackage{amsfonts, mathrsfs}
\usepackage{amsmath}
\usepackage[numbers, sort&compress]{natbib}
\usepackage{graphicx}
\usepackage{vmargin, enumerate}
\usepackage{setspace}
\usepackage{paralist}
\usepackage[pdftex,colorlinks=true,citecolor=linkgreen,linkcolor=linkgreen]{hyperref}
\usepackage{color}
\usepackage{wrapfig}

\newcommand{\e}{{\mathbf E}}
\newcommand{\p}[1]{{\mathbf P}\left(#1\right)}
\newtheorem{thm}{Theorem}
\newtheorem{lem}[thm]{Lemma}
\newtheorem{cor}[thm]{Corollary}
\newtheorem{prop}[thm]{Proposition}

\newcommand\cS{{\mathcal S}}
\newcommand\cT{{\mathcal T}}
\newcommand{\pran}[1]{\left(#1\right)}

\newcommand{\eqdist}{\ensuremath{\stackrel{\mathrm{d}}{=}}}

\newcommand{\psub}[2]{{ \mathbf P}_{#1}\left( #2 \right)}
\definecolor{darkgreen}{rgb}{0.05,0.45,0.1}
\definecolor{linkgreen}{rgb}{0,0.5,0}

\begin{document}

\title{The mixing time of the Newman--Watts small world}
\author{Louigi Addario-Berry \and Tao Lei}
\address{Department of Mathematics and Statistics, McGill University, 805 Sherbrooke Street West,
		Montr\'eal, Qu\'ebec, H3A 2K6, Canada}
\email{louigi@math.mcgill.ca}
\email{tao.lei@mail.mcgill.ca}
\date{January 19, 2012} %; revised ...
%\urladdrx{http://www.math.mcgill.ca/~louigi/}


%{60C05 (68P10,68W40)} %%{Primary: <subject>; Secondary: <subject>}
\begin{abstract}
``Small worlds'' are large systems in which any given node has only a few connections to other points, but possessing the property that all pairs of points are connected by a short path, typically logarithmic in the number of nodes. The use of random walks for sampling a uniform element from a large state space is by now a classical technique; to prove that such a technique works for a given network, a bound on the mixing time is required. However, little detailed information is known about the behaviour of random walks on small-world networks, though many predictions can be found in the physics literature. The principal contribution of this paper is to show that for a famous small-world random graph model known as the Newman--Watts small world, the mixing time is of order $\log^2 n$. This confirms a prediction of Richard Durrett, who proved a lower bound of order $\log^2 n$ and an upper bound of order $\log^3 n$.
\end{abstract}

\maketitle

\section{Introduction}\label{sec:intro}
	The small-world phenomenon is a catchy name for an important physical phenomenon that shows up throughout the physical, biological and social sciences.
	In brief, the term applies to large, locally sparse systems (usually, possessing only a bounded number of connections from any given point) which nonetheless exhibit good long-range connectivity in the sense that there are short paths between all points in the system. The Erd\H{o}s--R\'enyi random graph $G_{n,p}$, is perhaps the most mathematically famous model possessing small-world behaviour: when $p=c/n$ for $c>1$ fixed, the average vertex degree is $c$, and the diameter of the largest connected
component is \cite{riordan2010diameter}
\[
\frac{\log n}{\log c} + 2 \frac{\log n}{\log(1/c^*)} + O_p(1),
\]
where $c^* < 1$ satisfies $ce^{-c}=c^*e^{-c^*}$ and $O_p(1)$ denotes a random amount that remains bounded in probability as $n \to \infty$.

The Erd\H{o}s--R\'enyi random graph is unsatisfactory as a small-world model in two ways: first, the network does not satisfy full connectivity (a constant proportion of vertices lie outside of the giant component); second, the graph is locally tree-like -- for any fixed $k$, the probability that there is a cycle of length at most $k$ through a randomly chosen node node is $o(1)$. In real-world networks showing small world behaviour (social or business networks, gene regulatory networks, networks for modelling infectious disease spread, scientific collaboration networks, and many others -- the book \cite{newman2006structure} contains many interesting examples), full or almost-full connectivity is standard, and short cycles are plentiful. Several connected models have been proposed which in some respects capture the desired local structure as well as small-world behaviour, notably the Bollob\'as--Chung \cite{bollobas1988diameter}, Watts--Strogatz \cite{watts1998collective}, and Newman--Watts \cite{newman1999scaling,newman1999renormalization} models. These models are closely related -- all are based on adding sparse, long range connections to a connected ``base network'' which is essentially a cycle.

Understanding the behaviour of random walks on small-world networks remains, in general, a challenging open problem. Numerical and non-rigorous results for return probabilities \cite{jespersen2000relaxation}, relaxation times \cite{noh2004random}, spectral properties \cite{farkas01spectra}, hitting times \cite{jasch2001target, condamin2007first}, and diffusivity \cite{gallos2007scaling} appear in the physics literature, but few rigorous results are known. In \cite{durrett2007random}, Durrett considers the Newman--Watts small world, proving lower and upper bounds on the mixing time of order $\log^2 n$ and $\log^3 n$, respectively, and suggests that the lower bound should in fact be correct. The principle contribution of this paper is to confirm Durrett's prediction. (Theorem~\ref{thm:main}, below).
\begin{wrapfigure}{r}{0.3\textwidth}
\vspace{-0.2cm}
\centering
\includegraphics[width=0.3\textwidth]{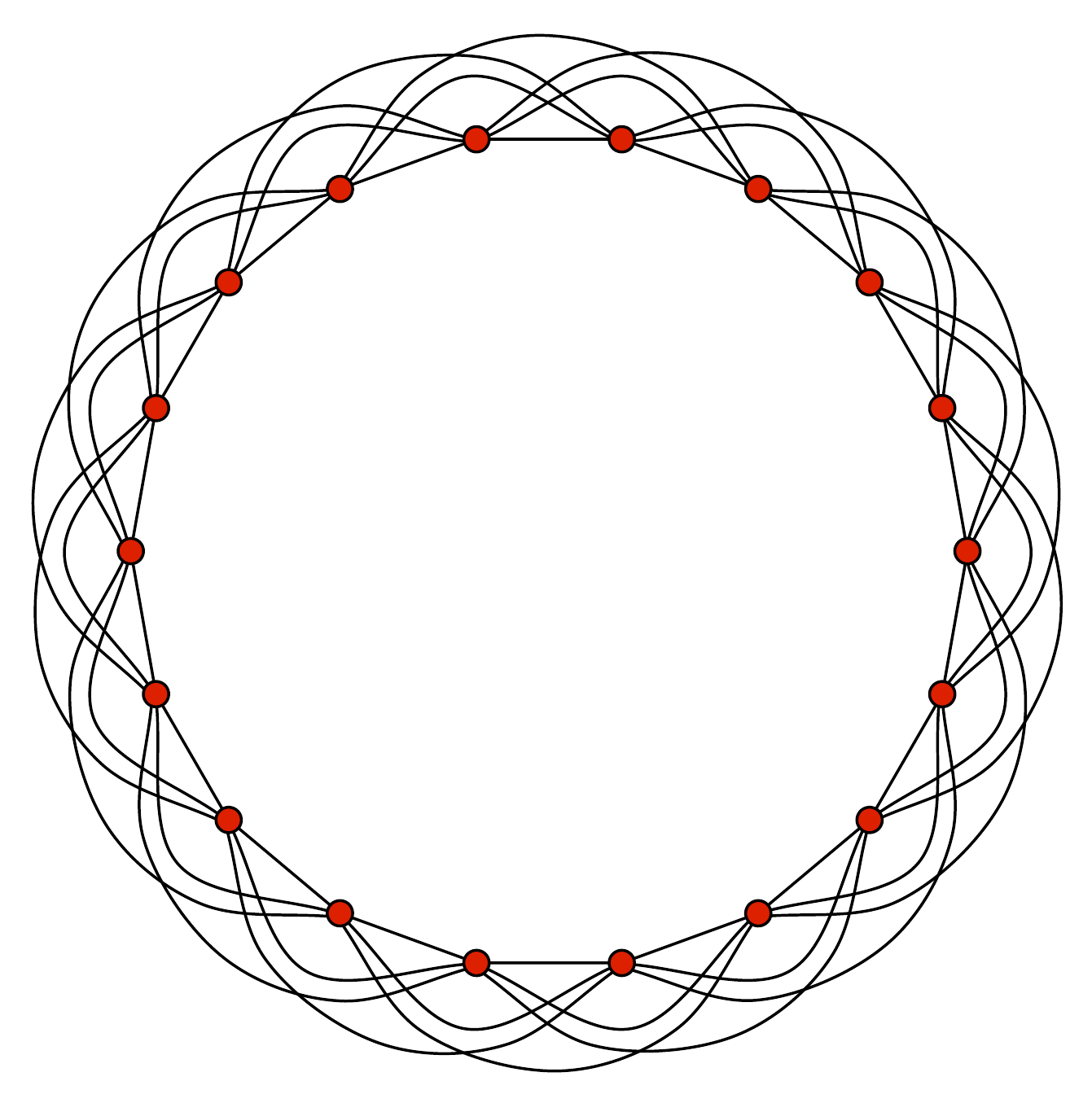}
%\caption{$R_{18,3}$}
\noindent
%\hrulefill
%\label{test}
\end{wrapfigure}

\vspace{-0.2cm}
To define the Newman--Watts small world, first fix integers $n \geq k \geq 1$. The {\em $(n,k)$-ring} $R_{n,k}$ is the graph with vertex set $[n]=\{1,\ldots,n\}$ and edge set
$\{\{i,j\}: i+1 \leq j \leq i+k\},$
where addition is interpreted modulo $n$. (A picture of $R_{18,3}$ appears to the right.) In particular, an $(n,1)$-ring is a cycle of length $n$, and whenever $n > 2k$ the $(n,k)$ ring is regular of degree $2k$.
For $0 < p < 1$, the $(n,k,p)$ Newman--Watts small world $H_{n,k,p}$ is the random graph obtained from the $(n,k)$-ring by independently replacing each non-edge of the $(n,k)$-ring by an edge with probability $p$. We write $H$ as shorthand for $H_{n,k,p}$ whenever the parameters are clear from context.

Given a (finite, simple) graph $G=(V,E)$, by a {\em lazy simple random walk on $G$} we mean a random walk that at each step stays still with probability $1/2$, and otherwise moves to a uniformly random neighbour. In other words, this is a Markov chain with state space $V$ and transition probabilities
\[
p_{x,y} = \begin{cases} \frac{1}{2} & \mbox{ if } x=y \\
					\frac{1}{2d_G(x)} & \mbox{ if } y\in N_G(x)\\
                     0 & \mbox{ otherwise }
		\end{cases}
\]
where $d_G(x)$ denotes the number of neighbours of $x$ in $G$ and $N_G(x)$ denotes the collection of neighbours of $x$ in $G$. (More generally, we shall say a chain is {\em lazy} if $p_{x,x}\ge 1/2$ for all $x$ in the state space.) If $G$ is connected then this Markov chain has a unique stationary distribution $\pi$ given by $\pi(x) = d_G(x)/2|E|$.

Given two probability distributions $\mu,\nu$ on $V$, we define the {\em total variation distance} between $\mu$ and $\nu$ as
\[
\|\mu-\nu\|_{\textsc{TV}} = \sup_{S \subset V} |\mu(S)-\nu(S)| = \frac{1}{2} \sum_{v \in V} |\mu(v)-\nu(v)|,
\]
where $\mu(S) = \sum_{v \in S} \mu(v)$. Now let $(X_k)_{k \geq 0}$ be a lazy simple random walk on $G$, and write $\mu_{k,x}$ for the distribution of $X_k$ when the walk is started from $x$; formally, for all $y \in G$, $\mu_{k,x}(y) = \p{X_k=y|X_0=x}$. The {\em mixing time} of the lazy simple random walk on $G$ is defined as
\[
\tau_{\textsc{mix}}(G) = \max_x\min\{k: \| \mu_{k,x} - \pi \|_{\textsc{TV}} \leq 1/4 \}.
\]
(There are many different notions of mixing time, many of which are known to be equivalent up to constant factors -- the book \cite{levin08markov}, and the survey \cite{lovasz1998mt} are both excellent references.) We may now formally state our main result.
\begin{thm}\label{thm:main}
Fix $c > 0$, let $p=c/n$, and let $H$ be an $(n,k,p)$ Newman--Watts small world. Then there is $C_0 >0$ depending only on $c$ and $k$ such that with probability at least $1-O(n^{-3})$,
\[
C_0^{-1} \log^2 n \leq \tau_{\textsc{mix}}(H) \leq C_0 \log^2 n.
\]
Furthermore, $\e[{\tau_{\textsc{mix}}(H)}] \leq C_0 (\log^2 n + 1)$.
\end{thm}
The expectation bound in Theorem~\ref{thm:main} follows easily from the probability bound. Indeed, given a finite reversible lazy chain $X=(X_t, t \ge 0)$ with state space $\Omega$, for $x \in \Omega$ write $\tau_x = \min\{t \ge 0:X_t=x\}$ for the hitting time of state $x$. Then $\tau_{\textsc{mix}} \le 2 \max_{x \in \Omega} \mathbb{E}_{\pi}(\tau_x)+1$, where $\mathbb{E}_{\pi}$ denotes expectation starting from stationarity (see, e.g., \cite{levin08markov}, Theorem 10.14 (ii)). If $X$ is a lazy simple random walk on a connected graph $G=(V,E)$ with $|V|=n$, then $\max_{x \in V}\mathbb{E}_{\pi}(\tau_x) \le (4/27+o(1))n^3$ (see \cite{brightwell1990maximum}).
Assuming the probability bound of Theorem~\ref{thm:main} and applying the two preceding facts we obtain that
\[
\e{[\tau_{\textsc{mix}}(H)]}\le C_0\log^2 n+\frac{8}{27}n^3(1+o(1))\cdot \p{\tau_{\textsc{mix}}(H) > C_0 \log^2 n} \le C_0(\log^2 n+1),
\]
assuming $C_0$ is chosen large enough.

The lower bound of Theorem~\ref{thm:main} is also straightforward, and we now provide its proof. For $n$ sufficiently large, given $v \in [n]$ and $\ell \in \mathbb{N}$ the probability that all vertices $w$ of $H_{n,k,p}$ with $w \in [v-\ell,v+\ell] \mod n$ have degree exactly $2k$, is greater than 
\[
\left(1-p \right)^{n (2\ell+1)} \ge e^{-2c(2\ell+1)},
\]
the inequality holding since $(1-c/n)^n > e^{-2c}$ for $n$ large. Taking $\alpha=1/(8c)$, it follows easily that with probability at least $1-O(n^{-3})$ there is $v \in [n]$ such that all vertices $w$ with $w \in [v-\alpha \log n,v+\alpha \log n] \mod n$ have degree exactly $2k$. Furthermore, the random walker started from such a vertex $v$ will with high probability take time of order $\log^2 n$ before first visiting a vertex in the complement of $[v-\alpha \log n,v+\alpha \log n] \mod n$. 
%(This may be easily shown by coupling the random walk starting from $v$ to a symmetric simple random walk on $\mathbb{Z}$.) 
Finally, under $\pi$, the set $[v-\alpha \log n,v+\alpha \log n] \mod n$ has measure tending to zero with $n$, and it thus follows from the definition of $\tau_{\textsc{mix}}$ that the mixing time is of order at least $\log^2 n$ whenever such a vertex $v$ exists. 

%Indeed, it is easily verified that a graph containing a connected set of $\Theta(\log n)$ vertices of degree $2$ has mixing time $\Omega(\log^2 n)$, which establishes the lower bound of Theorem~\ref{thm:main} in the case that $k=1$; the case $k>1$ is similarly straightforward.)
%\subsection{Plan of the paper}

Having taken care of the expectation upper bound and lower bound in probability from Theorem~\ref{thm:main}, the remainder of the paper is now devoted to proving that with probability at least $1-O(n^{-3})$, $\tau_{\textsc{mix}}(H) = O(\log^2 n)$. 
In Section~\ref{sec:mix} we explain a conductance-based mixing time bound of Fountoulakis and Reed \cite{fountoulakis2007faster} which will form the basis of our approach. 
The Fountoulakis--Reed bound requires control on the edge expansion of connected subgraphs of $H_{n,k,p}$. To this end, in Section~\ref{sec:lagrange} we bound the expected {\em number} of connected subgraphs of $H_{n,k,p}$ of size $j$, for each $1 \le j \le n$; our probability bounds  on the edge expansion of such subgraphs follow in Section~\ref{sec:expansion}. Finally, in Section~\ref{sec:proof} we finish the proof of Theorem~\ref{thm:main}. The proof is more straightforward when $c$ is large; in order to get the key ideas across cleanly we accordingly handle the large-$c$ and small-$c$ cases separately. 

\subsection{Notation}
Given a graph $G$, write $V(G)$ for the set of vertices of $G$, and $E(G)$ for the set of edges of $G$. Also, given $S \subset V(G)$, write $G[S]$ for the subgraph of $G$ induced by $S$. We say $S$ is connected if $G[S]$ is connected.
Finally, given a formal power series $F(z)$, we write $[z^j] F(z)$ to mean the coefficient of $z^j$ in $F(z)$, so if $F(z) = \sum_{k \geq 0} a_k z^k$ then $[z^j]F(z)=a_j$.

%\section{The Newman--Watts small-world}\label{nw}
%Fix integers $n \geq k \geq 1$. The {\em $(n,k)$-ring} $R_{n,k}$ is the graph with vertex set $[n]=\{1,\ldots,n\}$ and edge set
%\[
%\{\{i,j\}: i+1 \leq j \leq i+k\},
%\]
%where addition is interpreted modulo $n$. In particular, an $(n,1)$-ring is a cycle of length $n$, and whenever $n > 2k$ the $(n,k)$ ring is regular of degree $2k$.
%For $0 < p < 1$, the $(n,k,p)$ Newman--Watts small world $H_{n,k,p}$ is the random graph obtained from the $(n,k)$-ring by independently replacing each non-edge of the $(n,k)$-ring by an edge with probability $p$. We write $H$ as shorthand for $H_{n,k,p}$ whenever the parameters are clear from context.

\section{Mixing time via conductance bounds}\label{sec:mix}

A range of techniques are known for bounding mixing times (\cite{montenegro2006mathematical} is a recent survey of the available approaches), many of which are tailor-made to give sharp bounds for particular families of chains. One particularly fruitful family of techniques is based on bounding the {\em conductance} of the underlying graph, a function which encodes the presence of bottlenecks at all scales.
The precise bound we shall use is due to Fountoulakis and Reed \cite{fountoulakis2007faster}.
Given sets $S,T \subset V$, write $E(S,T)=E_G(S,T)$ for the set of edges of $G$ with one endpoint in $S$ and the other in $T$, and write $e(S,T)=|E(S,T)|$.
Also, given $S \subset V$ write $e(S) = \sum_{v \in S} d_G(v)$. The {\em conductance} of $S$, written $\Phi(S)$, is given by
\[
\Phi(S) = \frac{e(S,S^c)}{e(S)}.
\]
For $0 \leq x \leq 1/2$, write
\[
\Phi(x) = \mathop{\min_{S~\mathrm{connected}}}_{x|E| \leq e(S) \leq 2x|E|} \Phi(S);
\]
one can think of $\Phi(x)$ the (worst case) {\em connected conductance of $G$ at scale $x$}. We will use the following theorem, a specialization of the main result from \cite{fountoulakis2007faster}.
\begin{thm}[\cite{fountoulakis2007faster}]\label{thm:fr}
There exists a universal constant $C > 0$ so that for any connected graph $G$,
\[
\tau_{\textsc{mix}}(G) \leq C \sum_{i=1}^{\lceil \log_2 |E| \rceil} \Phi^{-2}(2^{-i}).
\]
\end{thm}
With this theorem at hand, proving mixing time bounds boils down to understanding what sorts of bottlenecks can exist in $G$. For the $(n,k,p)$ Newman--Watts small world with $p=c/n$, it is not hard to see that small sets can have poor conductance. Indeed, in the introduction we observed that with high probability the ring $R_{n,k}$ will contain connected sets $S$ with $\Theta(\log n)$ nodes, to which no edges are added in $H_{n,k,p}$. Such a set $S$ will have $e(S,S^c) <k^2$ and so will have conductance $\Phi(S) = O(1/\log n)$. It follows that in this case the best possible mixing time bound one can hope to prove using Theorem~\ref{thm:fr} is of order $\log^{2} n$. 

We will prove the upper bound in the probability bound of Theorem~\ref{thm:main} by showing that there are constants $\epsilon > 0$, $C_0> 0$ such that with high probability, whenever $|S| \geq C_0 \log n$, we have $\Phi(S) \geq \epsilon$. To accomplish this using Theorem~\ref{thm:fr}, we will need control on the likely number of connected subgraphs of $H_{n,k,p}$ of size $s$, for all $s \ge \log n$. In the next section, we bound the expected number of such subgraphs using Lagrange inversion and comparison with a branching process.

\section{Counting connected subgraphs}\label{sec:lagrange}

Fix $c > 0$ and $k \geq 1$, let $p=c/n$, and let $H=H_{n,k,p}$ be a Newman--Watts small world. Let $v\in [n]$, write $B_{j,v}=B_{j,v}(H)$ for the set of all $S \subset [n]$ containing $v$ with $|S|=j$ such that $H[S]$ is connected, and let $B_j=\bigcup\limits_{v\in [n]}B_{j,v}$. Our aim in this section is to establish the following proposition.

\begin{prop}\label{prop:lagrange1}

For any positive integer $j$ and any $v \in [n]$, $\e|B_{j,v}| \leq (4(c+2k))^j$, and $\e|B_j|\leq n(4(c+2k))^j$.

\end{prop}
We will prove Proposition~\ref{prop:lagrange1} by comparison with a Galton--Watson process.
Recall that a Galton--Watson process can be described as follows. An initial individual -- the progenitor -- has a random number $Z_1$ of children, where $Z_1$ is some non-negative, integer-valued random variable. The distribution of $Z_1$ is called the {\em offspring distribution}.
%The Galton-Watson process is described as the following: in a population, we begin with one person -- the root. The root has $Z_1$ children, where $Z_1$ follows a fixed distribution (called \emph{offspring distribution}).
Each child of the progenitor reproduces independently according to the offspring distribution, and this process continues recursively.
%$Z_i$ denotes the number of population in $i$-th generation (the root being 0-th generation). If we view each person as a vertex and draw an edge between each pair of parent and child, then this process can be visualized as a tree, the \emph{Galton-Watson tree}.
The family tree of a Galton--Watson process is called a Galton--Watson tree, and is rooted at the progenitor.

The number of neighbours of a vertex in $H_{n,k,p}$ is distributed as $\mathrm{Bin}(n-2k-1,p)+2k$. From this it is easily seen that $|B_{j,v}|$ is stochastically dominated by the number of subtrees of size $j$ containing the root in a Galton--Watson tree with offspring distribution $\mathrm{Bin}(n-2k-1,p)+2k$. To bound the expectations of the latter random variables, we will first encode these expectations as the coefficients of a generating function, then use the Lagrange inversion formula (\cite{stanley1999enumerative}, Theorem 5.4.2), which we now recall. This approach was suggested to us by Omer Angel in a mathoverflow comment (\url{http://mathoverflow.net/questions/66595/}); we thank him for the suggestion.

\begin{thm}[Lagrange inversion formula]\label{thm:lagrangeinversionformula}
If $G(x)$ is a formal power series and $f(x)=xG(f(x))$, then
$$n[x^n]f(x)^k=k[x^{n-k}]G(x)^n.$$
\end{thm}

Fix a non-negative, integer-valued random variable $B$, and for $m \geq 0$ write $p_m=\p{B=m}$. Given a Galton-Watson tree $\cT$ with offspring distribution $B$, let $\mu_j=\mu_j(B) $ denote the expected number of subtrees of $\cT$ containing the root of $\cT$ and having exactly $j$ vertices (so $\mu_0=0$). Also, write
\[
q_j =\sum_{m \geq j} p_m (m)_j,
\]
where $(m)_j = m!/(m-j)!$ is the falling factorial. Note that $q_j$ is the expected number of ways to choose and order precisely $j$ children of the root in $\cT$.
Let $F(z)=\sum_{j=0}^{\infty}\mu_j z^j$ and $Q(z)=\sum_{j=0}^{\infty}q_j z^j$ be the generating functions of $\mu_j$ and $q_j$ respectively, viewed as formal power series.

\begin{lem}\label{lem:lagrange1}
$F(z)=zQ(F(z))$
\end{lem}
\begin{proof}
We have
\begin{align}
    Q(F(z))=& \sum\limits_{j\ge 0}q_j\left(\sum\limits_{r\ge 1}\mu_rz^r\right)^j \nonumber \\
       		=& \sum\limits_{j\ge 0} q_j\left(\sum\limits_{r\ge j}\sum\limits_{\genfrac{}{}{0pt}{}{r_1+\cdots+r_j=r}{r_1, \cdots, r_j \in \mathbb{N}^+}}z^r \mu_{r_1}\cdots\mu_{r_j}\right) \nonumber \\
       		=& \frac{1}{z} \sum\limits_{r\ge 0}z^{r+1}\sum\limits_{j\le r}q_j\left(\sum\limits_{\genfrac{}{}{0pt}{}{r_1+\cdots+r_j=r}{r_1, \cdots, r_j \in \mathbb{N}^{+}}} \mu_{r_1}\cdots\mu_{r_j}\right) \label{onestep}
\end{align}
The $r$'th term in the outer sum in (\ref{onestep}) encodes subtrees of $\cT$ with $r+1$ vertices that contain the root, as follows. First specify the degree $j$ of the root of the tree $T$ to be embedded. Then choose which $j$ children of the root of $\cT$ will form part of the embedding, and the order in which the children of the root of $T$ will be mapped to these nodes (there are $q_j$ ways to do this on average). Next, choose the sizes $r_1,\ldots,r_j$ of the subtrees of the children of the root in the embedded tree; finally, embed each such subtree in the respective subtree of $\cT$; on average, there are $\mu_{r_i}$ ways to do this.
It follows that
\begin{align}
(\ref{onestep})    		=& \frac{1}{z} \sum\limits_{r\ge 0}z^{r+1}\mu_{r+1} \label{twostep} \\
       					=& \frac{1}{z}F(z),	\nonumber
\end{align}
which proves the lemma. (We remark that verifying that (\ref{onestep}) and (\ref{twostep}) are equal can be done purely formally; however, we find the preceding explanation more instructive.) \end{proof}

\begin{lem}\label{lem:lagrange2}
Fix $C > 0$. if $q_j \leq C^j$ for all $j \geq 0$ then $\mu_j \leq \frac{1}{j} {2j-2 \choose j-1} C^{j-1} < (4C)^{j-1}$ for all $j \geq 1$.	
\end{lem}
\begin{proof}
By Lemma~\ref{lem:lagrange1} and Theorem \ref{thm:lagrangeinversionformula}, we have $j[z^j]F(z)^k=k[z^{j-k}]Q(z)^j$. In particular, taking  $k=1$, we have $\mu_j=[z^j]F(z)=\frac{1}{j}[z^{j-1}]Q(z)^j$. Now,
\[
(Q(z))^j=\pran{\sum\limits_{l\ge 0}q_l z^l}^j=\sum\limits_{r\ge 0}\left(\sum\limits_{\genfrac{}{}{0pt}{}{l_1+\cdots+l_j=r}{l_1, \cdots, l_j\in\mathbb{N}}}q_{l_1}q_{l_2}\cdots q_{l_j}\right)z^r.
\]
Therefore $[z^{j-1}]Q(z)^j=\sum\limits_{\genfrac{}{}{0pt}{}{l_1+\cdots+l_j=j-1}{l_1, \cdots, l_j\in\mathbb{N}}}q_{l_1}q_{l_2}\cdots q_{l_j}$. Each summand $q_{l_1}\cdots q_{l_j}$ is at most ${C}^{j-1}$ by assumption. There are ${2j-2 \choose j-1}$ nonnegative integer solutions to the equation $l_1+\cdots+l_j=j-1$, so we obtain that $[z^{j-1}] Q(z)^j \leq {2j-2 \choose j-1} \cdot C^{j-1}$. The result follows.
\end{proof}
The next lemma controls the growth of $q_j$ for some important special offspring distributions, which allows us to use Lemma~\ref{lem:lagrange2} to prove Proposition~\ref{prop:lagrange1}.
\begin{lem}\label{lem:lagrange3}
If $B$ is Poisson$(c)$ distributed then $q_j=c^j$ for all $j$. Also, if $B$ is $\mathrm{Bin}(n,c/n)$ distributed, then for all $j \geq 0$, $q_j \leq c^j$.
Finally, if $B-\ell$ is $\mathrm{Bin}(n,c/n)$ distributed for some fixed $\ell \geq 0$, then for all $j \geq 0$, $q_j \leq (c+\ell)^j$.
\end{lem}
\begin{proof}
%First note that
%\[
%q_j=\sum_{m \ge j} p_m (m)_j \le \sum\limits_{m\ge j}p_m m^j=\mathbf{E}[B^j\mathbf{I}_{B \geq j}].
%\]
If $B$ is Poisson with mean $c$, then
\[
  q_j
   =\sum_{m \geq j} p_m \frac{m!}{(m-j)!}
   = e^{-c}c^j\left(1+\frac{c}{1!}+\frac{c^2}{2!}+\cdots\right)
   = e^{-c}c^je^c
   = c^j
\]
If $B\eqdist \mathrm{Bin}(n,c/n)$ then
\begin{align*}
q_j 		& = \sum_{j \leq m \leq n} {n \choose m} \pran{\frac{c}{n}}^m \pran{1-\frac{c}{n}}^{n-m} \frac{m!}{(m-j)!} \\
		& = \pran{\frac{c}{n}}^j \cdot n(n-1) \cdots (n-j+1) \sum_{j \leq m \leq n} {n-j \choose n-m}  \pran{\frac{c}{n}}^{m-j} \pran{1-\frac{c}{n}}^{n-m} \\
		& \leq c^j,
\end{align*}
Finally, if $B \eqdist \mathrm{Bin}(n,c/n)+\ell$ then we consider three Galton-Watson trees $\cT_1, \cT_2, \cT_3$ with offspring distributions $B_1\equiv l, B_2\eqdist\mathrm{Bin}(n,c/n)$ and $B_3 \eqdist \mathrm{Bin}(n,c/n)+\ell$ respectively. For $i=1,2,3$ write $q_j^{(i)}=q_j(B_i)$. Since $q_j^{(1)}$ is the expected number of ways to choose and order precisely $j$ children of the root in $\cT_1$, we have $q_j^{(1)}=(l)_j\le l^j$ and by the previous argument we know that $q_j^{(2)}\le c^j$. Finally, since $q_j^{(3)}$ is the expected number of ways to choose and order precisely $j$ children of the root in $\cT_3$, by independence we have
\begin{align*}
q_j^{(3)} & = \sum\limits_{s=0}^j{j \choose s}q_s^{(1)}q_{j-s}^{(2)} \\
& \le \sum\limits_{s=0}^j{j \choose s}l^sc^{j-s} \\
& = (l+c)^j.
\end{align*}

The factor ${j \choose s}$ in the first equation is because as long as we choose $s$ positions for children coming from the deterministic component of offspring distribution, the order of all $j$ children are fixed since the order among $s$ children and the order among the other $j-s$ are both fixed.
\end{proof}
We remark that if $\cT$ has deterministic $d$-ary branching (every node has exactly $d$ children with probability one), then for all $j$, the number of subtrees containing the root and having precisely $j$ nodes is exactly ${dj \choose j-1}/j$ (Thm 5.3.10 in \cite{stanley1999enumerative}) which is bounded above by $\left(\frac{edj}{j-1}\right)^{j-1}/j\le (ed)^j$. Thus, Lemma~\ref{lem:lagrange2} shows that when factorial moments grow only exponentially quickly, the values $\mu_j$ behave roughly as in the case of deterministic branching. We also note that when $\cT$ has Poisson$(c)$ branching distribution, Lemma~\ref{lem:lagrange2} and the argument of Lemma~\ref{lem:lagrange3} together yield the exact formula $\mu_j = \frac{c^{j-1}}{j} {2j-2 \choose j-1}$.
\begin{proof}[Proof of Proposition~\ref{prop:lagrange1}]
For any $v \in V(H_{n,k,p})$, the random variable $|B_{j,v}|$ is stochastically dominated by $\mu_j(B)$, where $B \eqdist \mathrm{Bin}(n,c/n)+2k$.
By Lemma~\ref{lem:lagrange3}, we have that $q_j(B) \le (c+2k)^j$ for all $j$. It then follows from Lemma~\ref{lem:lagrange2} that $q_j \le (4(c+2k))^{j-1}$ for all $j$, proving the proposition.
\end{proof}

%\begin{lem}\label{lem:lagrange}
% Then the following hold:
%\begin{itemize}
%\item[(i)] $F(z)=zQ(F(z))$;
%\item[(ii)] If $B$ is Poisson$(c)$ or $\mathrm{Bin}(n,c/n)$ distributed, then for all $j \geq 0$, $q_j \leq c^j$.
%Also, if $B-\ell$ is $\mathrm{Bin}(n,c/n)$ distributed for some fixed $\ell \geq 0$, then for all $j \geq 0$, $q_j \leq (c+\ell)^j$.
%\item[(iii)] Fix $C > 0$. if $q_j \leq C^j$ for all $j \geq 0$ then $\mu_j \leq \frac{1}{j} {2j-2 \choose j-1} C^{j-1} < (4C)^{j-1}$ for all $j \geq 1$.
%\end{itemize}
%\end{lem}

\section{Bounding the expansion of connected subgraphs of $H_{n,k,p}$}\label{sec:expansion}

Recall from the preceding section that $B_j$ is the collection of connected subsets $S$ of $V(H_{n,k,p})$ with $|S|=j$. We will show that with high probability, for all $j \ge \log n$, all elements $S$ of $B_j$ have conductance uniformly bounded away from zero. Many of our proofs are easier when $c$ is large, and we treat this case first.

%\begin{itemize}
%\item[1.] Galton-Watson trees.
%
%
%\item[2.] The Lagrange inversion formula.
%
%%We will use the following theorem known as the Lagrange inversion formula (\cite{stanley1999enumerative}, Theorem 5.4.2):
%%\begin{thm}\label{thm:lagrangeinversionformula}
%%If $G(x)$ is a formal power series and $f(x)=xG(f(x))$, then
%%$$n[x^n]f(x)^k=k[x^{n-k}]G(x)^n.$$
%%\end{thm}
%
%\item[3.] Chernoff bounds.

In the course of the proofs we will make regular use of the standard Chernoff bounds (see, e.g., \cite{janson00random} Theorem 2.1) which we summarize here:

\begin{thm}\label{thm:chernoff}
If $X \eqdist \mathrm{Bin}(m,q)$ then

for all $0 < x < 1$, $\p{X \leq (1-x)mq} \le \exp(-mq\phi(-x)) \le \exp(-mqx^2/2)$,

and for all $x > 0$,
$\p{X \geq (1+x)mq} \le \exp(-mq\phi(x)) \le  \exp(-mqx^2/2(1+x))$,

where $\phi(x)=(1+x)\log(1+x)-x$.
\end{thm}
We will use the coarser bounds most of the time. The finer bound will only be used twice, once in the proof of Lemma \ref{largesets} and once in the proof of Theorem \ref{thm:main}.

%\item[4.] The FKG inequality.

We will also have use of the FKG inequality (\cite{janson00random}, Theorem 2.12), which we now recall. Let $\Gamma=[n]=\{1, 2, \cdots, n\}$. Given $0\le p_1, \cdots, p_n\le 1$, $\Gamma_{p_1, \cdots, p_n} \subset [n]$ is obtained by including element $i$ with probability $p_i$ independently for all $i$. We say a function $f: 2^{\Gamma}\rightarrow \mathbb{R}$ is \emph{increasing} if $f(A)\le f(B)$ for $A\subset B$, $f$ \emph{decreasing} if $f(A)\ge f(B)$ for $A\subset B$.

\begin{thm}[FKG inequality]\label{thm:FKG}
If the random variables $X_1$ and $X_2$ are two increasing or two decreasing functions of $\Gamma_{p_1, \cdots, p_n}$, then
$$\mathbb{E}(X_1X_2)\ge\mathbb{E}(X_1)\mathbb{E}(X_2).$$
In particular, if $\mathcal{Q}_1$ and $\mathcal{Q}_2$ are two increasing or two decreasing families of subsets of $\Gamma$, then
$$\mathbb{P}\left(\Gamma_{p_1,\cdots, p_n}\in\mathcal{Q}_1\cap\mathcal{Q}_2\right)\ge\mathbb{P}\left(\Gamma_{p_1,\cdots,p_n}\in\mathcal{Q}_1\right)\mathbb{P}\left(\Gamma_{p_1,\cdots,p_n}\in\mathcal{Q}_2\right).$$
\end{thm}
We will in fact use the following, easy consequence of the FKG inequality.

\begin{cor}\label{cor:FKG}
If $C$ is an increasing event, $A$ is a decreasing event, then
$$\p{C ~|~ A }\le \p{C}.$$
\end{cor}

We begin by bounding the edge expansion of all but the very large connected sets, in the case that $c$ is large.
\begin{lem}\label{lem:eSSc}
Fix $c$ sufficiently large that $c/720-\log(4(c+2k)) > 5$. Then for all $n$,
\[
\p{\exists S \in \bigcup_{\log n \leq j \leq 9n/10} B_j, e(S,S^c) \le c|S|/12} \leq \frac{1}{n^3}.
\]
\end{lem}
\begin{proof}
Fix $j \in [\log n,9n/10]$, and $S \subset [n]$ with $|S|=j$. Note that $E(S,S^c)$ is independent of $H[S]$, so given that $S \in B_j$,
$e(S,S^c)$ stochastically dominates a $\mathrm{Bin}(j(n-j),p)$ random variable. Since $j(n-j) \ge jn/10$, it follows that under this conditioning $e(S,S^c)$ also stochastically dominates $X$, a $\mathrm{Bin}(nj/10,p)$ random variable. By a union bound it follows that
\begin{align*}
	& \quad \p{\exists S, |S|=j, H[S]\mbox{ connected}, e(S,S^c) \le cj/12} \\
\leq 	& \quad \sum_{S,|S|=j} \p{e(S,S^c) \le cj/12 ~|~ H[S]\mbox{ connected}} \cdot \p{H[S]\mbox{ connected}} \\
\leq	& \quad \p{X \leq cj/12} \cdot \e{|B_j|} \\
\leq 	& \quad e^{-cj/720} \cdot n(4(c+2k))^j,
\end{align*}
the last line by a Chernoff bound and by Proposition~\ref{prop:lagrange1}. (This is a typical example of our use of Proposition~\ref{prop:lagrange1} in the remainder of the paper.)

By our assumption that $c/720-\log(4(c+2k)) > 5$ and since $j \geq \log n$, we obtain that this probability is at most
\[
\exp\pran{\log n + j\pran{\log(4(c+2k))-\frac{c}{720}}}  \leq \frac{1}{n^4}.
\]
The result follows by a union bound over $j \in [\log n,9n/10]$.
\end{proof}
The next lemma provides a lower bound on the edge expansion of very large sets, again in the case that $c$ is sufficiently large.
\begin{lem}\label{9nover10}
If $c > 40$ then for all $n$ sufficiently large,
\[
\p{\exists S \subset [n]: |S| > 9n/10, e(S) \le |E(H)|} \le (2/e)^n.
\]
%is
%then for all $n$ sufficiently large, with probability at least $1-(2/e)^n$, for all $S \subset [n]$ with $|S| > 9n/10$ we have $e(S) \geq |E(H)|$.
\end{lem}
\begin{proof}
In this proof write $E=E(H)$. Since, for any set $S \subset [n]$, $e(S)+e(S^c)=2|E|$, it suffices to prove that
\[
\p{\exists S \subset [n]: |S| < n/10, e(S) \ge |E|} \le (2/e)^n.
\]
Fix any set $S$ with $|S| < n/10$. Write $e^*(S)= \sum_{v \in S} |\{e \ni v: e \not\in E(R_{n,k})\}|$ for the total degree incident to $S$ not including edges of the ring $R_{n,k}$, and similarly let $E^* = E \setminus E(R_{n,k})$.
Since $|S| < n/10 < n/2$, in order to have $e(S) \ge |E|$ we must in fact have $e^*(S) \ge |E^*|$.
Also, $e^*(S)$ is stochastically dominated by $\mathrm{Bin}(n^2/10,p)$, and $|E^*|\eqdist\mathrm{Bin}(n(n-1-2k)/2, p)$. When $n$ is large enough that $n-1-2k > 4n/5$, we thus have
\begin{align*}
\p{e(S) \ge |E(H)|} 	& \leq \p{|E^*| \le cn/5} + \p{ e^*(S) > cn/5 } \\
				& \leq \p{|E^*| \le c(n-1-2k)/4} + \p{ e^*(S) > cn/5 } \\
				& < \exp\pran{-c(n-1-2k)/16} + \exp\pran{-cn/40},
\end{align*}
by a Chernoff bound. For $n$ sufficiently large the last line is at most $2e^{-cn/40} < 2e^{-n}$, and the result follows by a union bound over all $S$ with $|S| \leq n/10$ (there are less than $2^{n-1}$ such sets).
\end{proof}
A similar but slightly more involved argument yields the following result, which will be useful for dealing with smaller values of $c$.
\begin{lem}\label{largesets}
For any $c > 0$ there is $\beta=\beta(c) > 0$ such that for all $n$ sufficiently large
\[
\p{\exists S \subset [n]: |S| > (1-\beta) n, e(S) \le |E(H)|} \le (1-\beta)^n.
\]
%, with probability at least $1-(1-\beta)^n$,
%for all $S \subset [n]$ with $|S| \geq (1-\beta)n$ we have $e(S) \geq |E(H)|$.
\end{lem}
\begin{proof}
As in the proof of Lemma~\ref{9nover10}, it suffices to prove that for some $\beta > 0$,
\[
\p{\exists S\subset [n]: |S| < \beta n, e(S) > |E|} \le (1-\beta)^n.
\]
Furthermore, since $\p{\exists S\subset [n]: |S| < \beta n, e(S) > |E|}$ decreases as $\beta$ decreases, it suffices to find
$\beta > 0$ and $\epsilon > 0$ such that for $n$ sufficiently large,
\[
\p{\exists S\subset [n]: |S| < \beta n, e(S) > |E|} = O(e^{-\epsilon n}).
\]
We fix $0 < \beta < 1/(3e)$ small enough that $1/(2\beta) - 8k/c > 1 + 1/(3\beta)$.
Additionally, recalling the function $\phi(x) = (1+x)\log(1+x)-x$ from Theorem~\ref{thm:chernoff},
we choose $\beta$ small enough that $\phi(1/(3\beta)) > \log(1/(3\beta))/(6\beta)$.
Finally, we assume $\beta < c/36$.

For any $S \subset [n]$ with $|S| < \beta n$ and with $e(S) > |E|$, defining $e^*(S)$ and $E^*$ as in the proof of Lemma~\ref{9nover10}, we then have
%If the claim is not true, then we can find set $S\subset [n]$ such that $|S|<\beta n, e(S)>|E(H)|$. If $e(S)>|E(H)|$, then
$e^*(S)\ge|E|-2k|S|\ge |E|-2k\beta n$.
\begin{align}
& \p{\exists S\subset [n]: |S| < \beta n, e(S) > |E|}\nonumber\\
\le~ & \p{\exists S\subset [n]: |S| < \beta n, e^*(S)\ge |E|-2k\beta n}\nonumber\\
\le~ & \p{|E| \le {n \choose 2}p-2k\beta n}+\p{\exists S \subset [n]: |S|<\beta n, e^*(S)> {n\choose 2}p-4k\beta n}.\label{eq:largesets1}
\end{align}
Since $|E|$ stochastically dominates $Bin({n \choose 2}, p)$, by a Chernoff bound we have
\begin{align}
\p{|E| \le {n \choose 2}p-2k\beta n} & \le \exp\left(-\frac{1}{2}{n\choose 2}p\left(\frac{2k\beta n}{{n \choose 2} p}\right)^2\right)\nonumber\\
& <  \exp\left(-\frac{4 k^2\beta^2}{c}\cdot n\right),\label{eq:largesets2}%\nonumber\\
%& \le  \exp\left(-\epsilon n\right)\label{eq:largesets2}
\end{align}
which handles the first summand in (\ref{eq:largesets1}).
%for some $\epsilon>0$ depending on $\beta$.
For the second summand, let $X$ be Binomial$(\beta n(n-1),p)$ distributed, and note that for
all $S \subset [n]$ with $|S| \le \beta n$, $e^*(S)$ is stochastically dominated by $X$.
Also, for $\beta < 1/3$ there are less than $2{n \choose \lfloor \beta n \rfloor}$ subsets of $[n]$ of size less than $\beta n$, and follows by a union bound that
\begin{equation}\label{eq:largesets3}
\p{\exists S \subset [n]: |S|<\beta n, e^*(S)> {n\choose 2}p-4k\beta n}
\le 2 {n \choose \lfloor \beta n \rfloor} \p{X > {n \choose 2}p - 4k\beta n}.
\end{equation}
Since $1/(p(n-1))=n/(c(n-1)) \le 2/c$ for all $n \ge 2$, and by our assumption that $1/(2\beta) - 8k/c > 1 + 1/(3\beta)$, we have
\[
{n \choose 2}p - 4k\beta n = \beta n(n-1)p \left(\frac{1}{2\beta} - \frac{4k}{p(n-1)}\right)
> \beta n(n-1) p \left(1+ \frac{1}{3\beta}\right).
\]
%for some $S$ with $|S|<\beta n$, we fix some small $\delta>0$ and pick $\beta$ small enough such that $4k\beta n<\delta {n\choose 2}p$, also note that $e^*(S)$ is stochastically dominated by $Bin(\beta n(n-1), p)=Bin(2\beta{n\choose 2}, p)$ and use
By the sharper of the Chernoff upper bounds in Theorem \ref{thm:chernoff} and by our assumption that $\phi(1/(3\beta)) > \log(1/(3\beta))/(6\beta)$, it follows that
\begin{align*}
\p{X > {n \choose 2}p - 4k\beta n}
&\le \exp\left(- \beta n(n-1)p\cdot \phi(1/(3\beta))\right)\\
&\le \exp\left(- \beta n(n-1)p \frac{\log(1/(3\beta))}{6\beta}\right) \\
& <  \exp\left(- \frac{c \log (1/(3\beta))}{12} n\right)
\end{align*}
for $n$ sufficiently large.

%\begin{align*}
%\p{e^*(S)\ge {n\choose 2}p-4k\beta n} & \le \p{e^*(S)\ge (1-\delta){n\choose 2}p}\\
%& \le \p{Bin(2\beta{n\choose 2}, p)\ge (1-\delta){n\choose 2}p}\\
%& \le \exp\left(-2\beta{n\choose 2}p\frac{1}{3\beta}\log(\frac{1}{3\beta})\right)\\
%& = \exp\left(-Cn\log(\frac{1}{3\beta})\right)
%\end{align*}
%where $C=c/3$ and the second last line is true since if we let $\Delta=\frac{(1-\delta){n\choose 2}p}{2\beta{n\choose 2}p}-1$, then $\Delta= \frac{1-\delta}{2\beta}-1\ge \frac{1}{3\beta}$ if we pick $\beta>0$ small enough.
Combined with (\ref{eq:largesets3}) this yields
\begin{align*}
& \p{\exists S \subset [n]: |S|<\beta n, e^*(S)> {n\choose 2}p-4k\beta n} \\
<~& 2 {n \choose \lfloor \beta n \rfloor} \exp\left(- \frac{c \log (1/(3\beta))}{12} n\right) \\
\le~& 2 \left(\frac{e}{\beta}\right)^{\beta n}\exp\left(- \frac{c \log (1/(3\beta))}{12}n\right) \\
=~& 2 \exp \left(n \left(\beta + \beta\log(1/\beta) - (c/12) \log(1/(3\beta))\right)\right) \\
<~& 2 \exp \left(-(c/36) n\right),
\end{align*}
where in the last inequality we used that $\beta+\beta\log(1/\beta) < 2\beta\log(1/\beta) <
(c/18)\log(1/\beta)$ and that $\log(1/(3\beta)) > 1$. Together with (\ref{eq:largesets1}) and (\ref{eq:largesets2}) we obtain
\[
\p{\exists S\subset [n]: |S| < \beta n, e(S) > |E|} \le \exp\left(-\frac{4 k^2\beta^2}{c}\cdot n\right)+ 2 \exp \left(-(c/36) n\right),
\]
which completes the proof.
%Since there are around ${n \choose \beta n}$ sets $S$ of size less than $\beta n$, and $${n \choose \beta n}\le \left(\frac{en}{\beta n}\right)^{\beta n}=\exp\left(n\left(\beta+\beta\log(\frac{1}{\beta})\right)\right).$$ Note that $\beta\log(\frac{1}{\beta})\rightarrow 0$ as $\beta \rightarrow 0$, and take union bound for the second summand, we can bound the total probability by $$\exp(-\epsilon n)+\exp(-(cn\log(\frac{1}{3\beta})-\beta n)).$$ Therefore we can find some $\beta_0>0$ and $\gamma>0$ such that $$\p{\exists S, |S|<\beta n ~|~ e(S)\ge|E(H)|}\le \exp(-\gamma n), ~\forall \beta \le \beta_0.$$ Take $\beta$ small enough such that $\exp(-\gamma n)\le (1-\beta)^n$ will suffice.
\end{proof}

In order to use Lemma~\ref{lem:eSSc} to bound the conductance of connected subsets $S$ of $V(H_{n,k,p})$ of size at most $9n/10$, we need to know that for such subsets we have $e(S) = O(|S|)$ with high probability. Such a bound is provided by the following lemma.

For given $k$, let $x=x_k$ be the positive solution of equation $x/720-\log(4(x+2k))=5$, and let $M=M(c,k) = k+1+10\max(x_k,c)$.
We remark that $x_k \geq 40$ for all $k \geq 1$.
\begin{lem}\label{lem:eSS}
For all $c > 0$ and for all $n$,
\[
\p{\exists S \in \bigcup_{1 \leq j \leq n} B_j, e(S,S) > M\cdot \max(|S|,\log n)} \leq \frac{1}{n^3}.
\]
\end{lem}
\begin{proof}
First note that for {\em fixed $M$}, the event whose probability we aim to bound is increasing, so increasing $p$ only increases its probability of occurrence. Since $M(c,k)$ is constant for $c \leq x_k$, it thus suffices to prove the bound for $c \geq x_k$, and we now assume that $c \geq x_k$. Note that in this case $c/720-\log(4(c+2k)) \geq 5$. For all $n$, and any $j \in [n]$, we have
\begin{align}
& \quad \p{ \exists S \in B_j, e(S,S) > (k+1+10c)\max(j,\log n)} \nonumber\\
\leq & \quad \sum_{S \subset [n], |S|=j} \p{e(S,S) > (k+1+10c)\max(j,\log n), S \in B_j}.\label{eq:eSS0}
\end{align}
Write ${\bf T}_S$ for the set of all possible trees on vertex set $S$ -- so $|{\bf T}_S|=|S|^{|S|-2}$ -- and list
the elements of ${\bf T}_S$ as $t_1,\ldots,t_r$.
For $i \in [r]$, let $F_i$ be the event that $t_i$ is a subgraph of $H$, and let $E_i=F_i \setminus \bigcup_{j < i} F_j$ be the event that $t_i$ is a subgraph of $H$ but none of $t_1,\ldots,t_{i-1}$ are subgraphs of $H$. The
events $E_i$ partition the event that $S \in B_j$, so
\begin{align}
			& \quad \p{e(S,S) > (k+1+10c)\max(j,\log n)~|~S \in B_j}\nonumber\\
 		\leq	&  \quad \max_{i \in [r]}\p{e(S,S) > (k+1+10c)\max(j,\log n)~|~E_i} \label{eq:eSS1}
%		\leq	&  \quad \max_{i \in [r]}\p{e(S,S) > (k+1+10c)\max(j,\log n)~|~t_i\mbox{ is a subgraph of}\ H}, \\
\end{align}
For fixed $i \in [r]$, write $\psub{i}{\cdot}$ for the conditional probability measure
$\p{\cdot~|~ F_i}$, and write $\Gamma^{(i)} = \{ uv: u,v \in S\}\setminus E(t_i)$.
Under $\mathbf{P}_i$, the set $E(S,S)\setminus E(t_i)$ is distributed as a Binomial$(p)$ random subset of $\Gamma^{(i)}$ since, after conditioning that $t_i$ is a subgraph of $H$, those edges not in $t_i$ still appear independently.

Write $C$ for the event that $|E(S,S)\setminus E(t_i)| > |(k+1+10c)\max(j,\log n)| - (j-1)$. Then
\begin{equation}\label{eq:eSS2}
\p{e(S,S) > (k+1+10c)\max(j,\log n)~|~E_i}
 = \psub{i}{C~\left|~ \bigcap_{j < i} F_j^c\right.}.
%& = \frac{\p{e(S,S) > (k+1+10c)\max(j,\log n),F_j,F_1,\ldots,
\end{equation}
Since $C$ is increasing and $\bigcap_{j < i} F_j^c$ is decreasing, it follows from Corollary \ref{cor:FKG} that
\[
(\ref{eq:eSS2}) \le \psub{i}{C} = \p{|E(S,S)\setminus E(t_i)| > (k+1+10c)\max(j,\log n) - (j-1)},
\]
the last equality holding since $E(S,S)\setminus E(t_i)$ is disjoint from $E(t_i)$, so independent of $F_i$.  Now, $|E(S,S)\setminus E(t_i)|$ is stochastically dominated by $kj+\mathrm{Bin}(\max(j,\log n) \cdot n/2,p)$. Letting $X$ have distribution $\mathrm{Bin}(n\max(j,\log n)/2, p)$, respectively, for all $i \in [r]$ we thus have
%since if we let $C$ be the event that $\{e(S,S) > (k+1+10c)\max(j,\log n)\}$, $A$ be the event that
%none of $t_1,\ldots t_{i-1}$ are subgraphs of $H$ and $B$ be the event that $t_i$ is a subgraph of $H$, then $C$ is an increasing event and $A$ is a decreasing event. And if we denote the new probability measure conditioning on $B$ by $\mathbf{P}^\ast$, i.e. $\mathbf{P}^\ast\{E\}=\p{E ~|~ B}$. Then $\mathbf{P}^\ast$ is still a product measure since conditioning on $t_i$ is a subgraph of $H$, those edges not in $t_i$ still appear independently. Therefore by Corollary \ref{cor:FKG}, we have $$\p{C ~|~ A, B}=\frac{\p{C, A ~|~ B}}{\p{A ~|~ B}}=\frac{\mathbf{P}^\ast\{C, A\}}{\mathbf{P}^\ast\{A\}}=\mathbf{P}^\ast\{C ~|~ A\}\le \mathbf{P}^\ast\{C\}=\p{C ~|~ B}.$$
%For any tree $t \in {\bf T}_S$,
%conditional on the event that t is a subgraph of $H$, we have that e(S,S) is stochastically dominated by $kj + (j-1)+  \mathrm{Bin}(jn/2,p)$, so by a Chernoff bound we obtain (let $X\sim\mathrm{Bin}(jn/2, p), Y\sim\mathrm{Bin}(n\max(j,\log n)/2, p)$)
%
\begin{align*}
& \quad \p{e(S,S) > (k+1+10c)\max(j,\log n)~|~E_i} \\
\leq & \quad \p{kj+(j-1)+X>(k+1+10c)\max(j, \log n)} \\
%\leq & \quad \p{kj+(j-1)+Y>(k+1+10c)\max(j, \log n)} \\
\leq & \quad \p{X>10c\max(j,\log n)}\\
\leq & \quad e^{-9c\max(j, \log n)/2},
\end{align*}
the last inequality holding by a Chernoff bound. It then follows from (\ref{eq:eSS0}) and (\ref{eq:eSS1}) that
\begin{align*}
	& \quad \p{ \exists S \in B_j, e(S,S) > (k+1+10c)\max(j,\log n)} \\
\leq 	& \quad \sum_{S \subset [n], |S|=j} e^{-9c\max(j, \log n)/2} \p{S \in B_j} \\
= & \quad e^{-9c\max(j, \log n)/2} \e{|B_j|} \\
\le & \quad e^{-9c\max(j, \log n)/2} n(4(c+2k))^j
\end{align*}
the last inequality by Proposition~\ref{prop:lagrange1}.
By assumption, $c$ is large enough that $c/720-\log(4(c+2k))> 5$, and it follows that
\begin{align*}
& \quad \p{ \exists S \in B_j, e(S,S) > (k+1+10c)\max(j,\log n)} \\
\le &  \quad e^{\log n + \log(4(c+2k))\max(j,\log n)-9c\max(j, \log n)/2} \\
<  & \quad n^{-4}.
\end{align*}
A union bound over $j \in [n]$ completes the proof.
%Since we have assumed that $c \ge x_k$, we obtain
%The remainder is similar to Lemma~\ref{lem:eSSc}: we can get $1/n^4$ bounds for both $j\ge\log n$ and $j<\log n$ cases, taking a union bound over $j\in [1, 9n/10)$ gives the result.
\end{proof}
\section{Proof of Theorem \ref{thm:main}}\label{sec:proof}
As noted in the introduction, the case when $c$ is large is more straightforward, and we handle it first.
\begin{proof}[Proof of Theorem \ref{thm:main} assuming $c>x_k$.]
For $x \ge 0$ write
\[
\Phi_0(x) =\min\left\{\frac{e(S,S^c)}{e(S)}:S~\mathrm{connected}, |S| \leq \frac{9n}{10}, x n(c/2 +k)/2 \leq e(S) \leq 4x n(c/2 +k)\right\}.
\]
Let $A$ be the event  that for all $S$ with $|S| > 9n/10$ we have $e(S) > |E(H)|$. If $A$ occurs then for all $0 \le x \le 1/2$, for all $S$ with $e(S) \le 2x|E(H)|$ we have $e(S) \le |E(H)|$ so $|S| \le 9n/10$. Also, let $A'$ be the event that $n(c/2+k)/2 \le |E(H)| \le 2n(c/2+k)$. If $A'$ occurs then for all $x\ge 0$ all $S$ with $x|E(H)| \le e(S) \le 2x|E(H)|$, we have $xn(c/2+k)/2 \le e(S) \le 4xn(c/2+k)$. It follows that on $A \cap A'$, for all $0 \le x \le 1/2$ we have $\Phi_0(x) \le \Phi(x)$.

We have $|E(H)| \eqdist nk+\mathrm{Bin}(n(n-2k-1)/2,p)$, so by a Chernoff bound, for all $n$ large enough, $\p{A'} \ge 1-n^{-3}$. Also, by Lemma \ref{9nover10}, $\p{A} \ge 1-n^{-3}$ for all $n$ sufficiently large.
%with probability at least $1-n^{-3}$ we have that $|S| \leq 9n/10$ for all $S \subset [n]$ with $e(S) \leq |E(H)|$.
It follows that with probability at least $1-2n^{-3}$, for all $0 \leq x \leq 1/2$ we have $\Phi_0(x) \leq \Phi(x)$, and thus with probability at least $1-2n^{-3}$,
\[
\sum_{i=1}^{\lceil \log_2 |E| \rceil} \Phi^{-2}(2^{-i}) \leq \sum_{i=1}^{\lceil \log_2 |E| \rceil} \Phi_0^{-2}(2^{-i}).
\]
We thus focus on bounding the latter quantity. Note that the sets considered when bounding $\Phi(2^{-i})$ decrease (in terms of $e(S)$) as $i$ increases.
Also, recall the definitions of $x_k$ and $M=M(c,k)$ from just before the statement of Lemma~\ref{lem:eSS}.

First suppose $i \geq \lfloor \log_2 (n(c/2+k)/(8M\log n))\rfloor$ and write
\[
i=\lfloor \log_2 (n(c/2+k)/(8M\log n))\rfloor + j.
\]
For all $S$ considered when bounding $\Phi_0(2^{-i})$ we have $e(S) \leq 2^{-j+6}M\log n$ and $e(S,S^c) \geq 2$, so
\begin{align*}
& \quad \sum_{j=1}^{\lceil \log_2 |E|\rceil-\lfloor \log_2 (n(c/2+k)/(8M\log n))\rfloor} \Phi_0^{-2}(2^{-\lfloor \log_2 (n(c/2+k)/(8M\log n))\rfloor-j}) \\
\leq & \quad 4^{5}M^2\log^2 n \sum_{j=0}^{\infty} 4^{-j} = \frac{ 4^{6}M^2}{3}\log^2 n.
\end{align*}

Next suppose that $i \leq \log_2 (n(c/2+k)/(8M\log n))$. In this case we have $e(S) \geq 4M  \log n$. By Lemma~\ref{lem:eSS}, with probability at least $1-n^{-3}$, for all connected sets $S$ with $|S| \le \log n$ we have $e(S,S) \le M\log n$. Since $e(S) \geq 4M \log n$ this implies that $e(S,S^c) =e(S)-2e(S,S)\geq e(S)-2M \log n \geq e(S)/2$, so $\Phi(S) \geq 1/2$. Also, by Lemmas~\ref{lem:eSSc} and~\ref{lem:eSS}, with probability at least $1-2n^{-3}$, for all connected sets $S$  with $|S| \geq \log n$, we have
\[
\frac{e(S,S^c)}{e(S)} = \frac{e(S,S^c)}{e(S,S^c) + 2 e(S,S)} \geq \frac{c|S|}{12}\frac{1}{c|S|/12 + 2M|S|} \geq \frac{c}{36M},
\]
so $\Phi(S) \geq c/(36M)$. Since $1/2 > c/(36M)$ it follows that with probability at least $1-3n^{-3}$, for all $i \leq \log_2 (n(c/2+k)/(8M\log n))$ we have $\Phi_0(2^{-i}) \geq \frac{c}{36M}$, and in this case
\[
\sum_{i=1}^{\lfloor \log_2 (n(c/2+k)/(8M\log n))\rfloor} \Phi_0^{-2}(2^{-i}) \leq \frac{6^4 M^2}{c^2} \log_2 n.
\]
Combining these bounds, we see that with probability at least $1-3n^{-3}$,
\[
\sum_{i=1}^{\lceil \log_2|E| \rceil} \Phi_0^{-2}(2^{-i}) \leq \frac{ 4^{6}M^2}{3}\log^2 n + \frac{6^4 M^2}{c^2} \log_2 n,
\]
so with probability at least $1-5n^{-3}$, $\sum_{i=1}^{\lceil \log_2|E| \rceil} \Phi^{-2}(2^{-i})$ is at most the same quantity. By Theorem~\ref{thm:fr} it follows that with probability at least $1-5n^{-3}$,
\[
\tau_{\textsc{mix}}(G) \leq C \pran{\frac{ 4^{6}M^2 \log^2 n}{3} + \frac{6^4 M^2}{c^2} \log_2 n}.
\]
This completes the proof in the case $c > x_k$. \end{proof}

For the remainder of the paper, fix $0 < c < x_k$, and let $R = \lceil \max(k,2x_1/c) \rceil$.
Also, recall the constant $\beta=\beta(c)$ from Lemma~\ref{largesets}. The remaining case of Theorem~\ref{thm:main} follows straightforwardly from the following lemma.
\begin{lem}\label{lem:smallc}
There is $\alpha=\alpha(c) > 0$ such that for all $n$ sufficiently large,
\[
\p{\exists S  \in \bigcup_{R\log n \leq j \leq (1-\beta)n} B_j(H), e_H(S,S^c) \leq \alpha |S|} \leq \frac{3R^3}{n^3}.
\]
\end{lem}
We provide the proof of Lemma~\ref{lem:smallc} at the end of the paper.
\begin{proof}[Proof of Theorem \ref{thm:main} assuming $c \le x_k$.]
For $x \ge 0$ write
\[
\Phi_0(x) =\min\left\{\frac{e(S,S^c)}{e(S)}:S~\mathrm{connected}, |S| \leq (1-\beta) n, x n(c/2 +k)/2 \leq e(S) \leq 4x n(c/2 +k)\right\}.
\]
As in the case $c > x_k$, by a Chernoff bound and by Lemma~\ref{largesets}, for all $n$ sufficiently large, with probability at least $1-2n^{-3}$ we have
\[
\sum_{i=1}^{\lceil \log_2 |E| \rceil} \Phi^{-2}(2^{-i}) \leq \sum_{i=1}^{\lceil \log_2 |E| \rceil} \Phi_0^{-2}(2^{-i}),
\]
and the remainder of the proof is just as in the case $c > x_k$, but using Lemma~\ref{lem:smallc} in place of Lemma~\ref{lem:eSSc}.
\end{proof}
It thus remains to prove Lemma~\ref{lem:smallc}; before doing so, we briefly describe our approach. We shall divide vertices of $H$ into groups of size $R$, each containing $R$ consecutive vertices. We view each group as a new single vertex; two new vertices are connected if there is an edge connecting their constituent sets. This yields an auxiliary graph $H'$, whose distribution is that of an $(n',1,p')$ Newman--Watts small world, for suitable $n'$ and $p'$. We will shortly see that $p'=c'/n'$ for some $c' > x_1$, so all ``large-$c$'' results can be applied to $H'$.

To translate edge expansion results from $H'$ into corresponding results for $H$, we proceed as follows. Given a set $S$ of vertices of $H$, we consider the \emph{blow-up} $S^+$ of $S$, which is the collection of all vertices of $H$ belonging to the same group as some element of $S$. The idea is that in most cases, the event that $e(S,S^c)$ is small relative to $|S|$ should be nearly identical to the event that $e(S^+,(S^+)^c)$ is small relative to $|S^+|$. If this were always true, Lemma \ref{lem:eSSc} would then yield bounds for the number of edges leaving $S^+$, which would in turn yield strong bounds on the probability that $e(S^+,(S^+)^c)< \epsilon ~ |S^+|$, where $\epsilon>0$ will be a function of $R$.

The above line of argument relies upon the intuition that the size of the blow-up $S^+$ should be essentially a constant factor greater than that of $S$. Since the ratio $|S^+|/|S|$ is in fact a random quantity, to make the above argument work, we end up needing to additionally show that $e(S^+,(S^+)^c)$ is very unlikely to be large if $e(S,S^c)/|S|$ is extremely small. In order to quantify the notion of ``extremely small'', we are forced to introduce a third parameter $\delta >0$ with $\delta$ much smaller than $\epsilon$. We now turn to the details.

\begin{proof}[Proof of Lemma~\ref{lem:smallc}]
%Fix an integer $R > k$ large enough that $Rc/2 > x_1$.
We assume for simplicity that $R$ divides $n$ -- the general case is practically identical -- and write $n'=n/R$. For $i \in [n']$ let $w_i = \{(i-1)R+j, 1 \leq j \leq R\}$. We form an auxiliary graph $H'=(V',E')$ with $V'=\{w_i,i \in [n']\}$ by adding an edge between $w_i$ and $w_j$ if there is some edge from an element of $w_i$ to an element of $w_j$ in $H$. It is easily verified (using the fact that $R>k$) that $H'$ is an $(n',1,p')$ Newman--Watts small world, with $p'=\p{Bin(R^2,p)>0} > Rc/2n'=c'/n'$ (where $c'=Rc/2$) for all $n$ sufficiently large. Note that since $c'=Rc/2 > x_1$, it follows that we may apply Lemma~\ref{lem:eSSc} to $H'$ -- this is the only way we will use this bound on $R$.

Given $S \subset [n]$, write $I' = \{i \in n': w_i \cap S \neq \emptyset\}$, let $S' = \{w_i, i \in I'\}$, and write $S^+ = \bigcup_{i \in S'} w_i$.
Now write $j=|S|$. As in the proof of Lemma~\ref{lem:eSS}, we partition the event that $S \in B_j(H)$ into $E_1(S),\ldots,E_r(S)$ according to the first spanning tree appearing in $S$. Fix $\epsilon = c/(12R(2Rc+1))$ and let $\delta>0$ be small enough that $\epsilon c\ge 2k\delta$ and that $\epsilon Rc\left(\log\left(\frac{\epsilon}{\delta}\right)-1\right)\ge 5+ \log(4(c+2k))$.  For all $1 \leq i \leq r$ we then have
\begin{align*}
& \quad \p{e_H(S^+\setminus S, (S^+)^c) > 2\epsilon Rc j ~|~ e_H(S,S^c) \leq \delta j, E_i(S)} \\
\leq & \quad \p{e_H(S^+\setminus S, (S^+)^c) > (\epsilon c+2k\delta) Rj ~|~ e_H(S,S^c) \leq \delta j, E_i(S)} \\
\leq & \quad \p{e_H(S^+\setminus S, (S^+)^c) > (\epsilon c+2k\delta) Rj ~|~ e_H(S,S^c) \leq \delta j, t_i \subset H}.
\end{align*}
The first inequality is true since we pick $\delta$ such that $\epsilon c\ge 2k\delta$ and the second inequality holds by Corollary \ref{cor:FKG}.
Now write $g(S) = |\{i \in S': |S \cap w_i| < |w_i|\}|$, so $g(S)$ is the number of sets $w_i$ that intersect $S$ but are not covered by $S$. It is easily checked that $e_H(S,S^c) \geq g(S)$; the extremal case is that for each $i \in S'$, $e_H(S \cap w_i,w_i \setminus S)=1$, while for $i,j \in S'$ with $i \ne j$, $e_H(w_i,w_j)=0$. %that each $w_i$ counted by $g(S)$ contributes $1$ for $e_H(S,S^c)$ from inside itself, while between $w_i$s there is no contribution to $e_H(S,S^c)$.

Next, suppose that $S \subset [n]$ satisfies $e_H(S,S^c) \leq \delta j$. Then we must have $g(S) \le \delta j$, and it follows that $|S^+\setminus S| \leq R \delta j$. Under this conditioning, $e_H(S^+\setminus S,(S^+)^c)$ is stochastically dominated by $2kR\delta j+\mathrm{Bin}(R \delta jn,p)$, and is independent of the event that $t_i \subset H$ since they are determined by disjoint sets of edges, so
by the finer of the Chernoff upper bounds,
\begin{equation*}\label{fewedges}
\p{e_H(S^+\setminus S,(S^+)^c) > (\epsilon c +2k\delta) Rj ~|~e_H(S,S^c) \leq \delta j, t_i \subset H} \leq \exp\pran{-\epsilon Rc j \left(\log\left(\frac{\epsilon}{\delta}\right)-1\right)}.
\end{equation*}
It follows that for $j \ge R \log n$, writing
\[
\cS_1 = \{S \subset [n], S \in B_j(H), e_H(S,S^c) \leq \delta j, e_H(S^+\setminus S,(S^+)^c) > 2\epsilon Rcj\},
\]
we have
\begin{align*}
\e|\cS_1| & \leq \sum_{S \subset [n], |S|=j} \sum_{1 \leq i \leq r}  \p{E_i(S), e_H(S,S^c) \leq \delta |S|, e_H(S^+\setminus S, (S^+)^c) > 2\epsilon Rc |S|} \\
	& \leq \exp\pran{-\epsilon Rc j \left(\log\left(\frac{\epsilon}{\delta}\right)-1\right)} \e\left|\{S \subset [n], |S|=j, H[S]\mbox{ connected}\}\right| \\
	& \leq \exp\pran{-\epsilon Rc j \left(\log\left(\frac{\epsilon}{\delta}\right)-1\right)} \cdot n \cdot (4(c+2k))^j \\
	& \leq n^{-4}
\end{align*}
the second-to-last inequality by Proposition~\ref{prop:lagrange1}, and the last since we chose $\delta$ such that
\[
\epsilon R c \left(\log\left(\frac{\epsilon}{\delta}\right)-1\right) \geq 5 + \log(4(c+2k)).
\]
and since $j\ge R\log n \ge \log n$.
It follows by a union bound over $R \log n \leq j \leq 9n/10$ and Markov's inequality that
\begin{equation}\label{wrapup1}
\p{\exists S  \in \bigcup_{R \log n \leq j \leq 9n/10} B_j(H), e_H(S,S^c) \le \delta |S|, e_H(S^+\setminus S,(S^+)^c) > 2\epsilon Rc|S|} \leq \frac{1}{n^3}.
\end{equation}
Next, write
\[
\cS_2 = \left\{ S  \in \bigcup_{R \log n \leq j \leq 9n/(10R)} B_j(H), e_H(S,S^c) \le \epsilon |S|, e_H(S^+\setminus S,(S^+)^c) \leq 2\epsilon Rc|S|\right\}.
\]
For any $S \in \cS_2$, we have
\begin{align*}
e_H(S^+,(S^+)^c) & =e_H(S, (S^+)^c)+e_H(S^+\backslash S, (S^+)^c) \\
				& \le e_H(S,S^c)+e_H(S^+\backslash S, (S^+)^c) \\
				& \le \epsilon (2Rc+1) |S| \\
				& \le \epsilon(2Rc+1)R|S'|\\
				& = c'|S'|/12.
\end{align*}
the last equality by our choice of $\epsilon$.  It follows that $e_{H'}(S',(S')^c) \leq c'|S'|/12$.
Furthermore, since $|S| \geq R \log n$ we have $|S'| \geq \log n \ge \log n'$, and since $|S| \leq 9n/(10R)$ we have $|S'| \leq 9n/(10R) = 9n'/10$.
It follows by Lemma~\ref{lem:eSSc} that
\begin{eqnarray}\label{wrapup2}
   \nonumber&& \p{\exists S  \in \bigcup_{R \log n \leq j \leq 9n/(10R)} B_j(H), e_H(S,S^c) \le \epsilon |S|, e_H(S^+\setminus S,(S^+)^c) \leq 2\epsilon Rc|S|} \\
   &\le& \p{\exists S' \in \bigcup_{\log n' \leq j\leq 9n'/10} B_j(H'), e_{H'}(S', (S')^c)\le c'|S'|/12} \leq \frac{1}{(n')^{3}}
\end{eqnarray}

%\begin{equation}\label{wrapup2}
%\p{\exists S  \in \bigcup_{R \log n \leq j \leq 9n/(10R)} B_j(H), e_H(S,S^c) \le \epsilon |S|, e_H(S^+\setminus S,(S^+)^c) \leq 2\epsilon Rc|S|} \leq \frac{1}{(n')^{3}}.
%\end{equation}
Next, for any $m \ge 1$, if $e_H(S, S^c) \le m$ then, viewed as a subset of a cycle of length $n$, $S$ must have at most $m$ connected components. The number of subsets of an $n$-cycle with at most $m$ connected components is $2 n{n+2m-1 \choose 2m-1}$. (This is a straightforward combinatorial exercise but may be seen as follows: the factor $n$ chooses a starting point on the cycle, the factor ${n+2m-1 \choose 2m-1}$ chooses the points on the cycle at which membership in $S$ alternates, and the factor $2$ accounts for whether or not the starting point belongs to $S$.)

It follows that for any $\gamma > 0$, the number of subsets of an $n$-cycle with at most $\gamma n$ connected components %(still when viewed as a subset of an $n$-cycle). The number of subsets of an $n$-cycle with at most $\gamma n$ connected components is at most $2n{n+2\gamma n-1 \choose 2\gamma n-1}$, where the factor ${n+2\gamma n-1 \choose 2\gamma n-1}$ is the number of nonnegative integer solutions to $x_1+x_2+\cdots+x_{2m}=n$ for $m=\gamma n$, factor $n$ accounts for the choice of starting point on the cycle and factor $2$ accounts for whether to put the block containing starting point in $S$ or not.
is at most
\begin{align}
2n{n+\lfloor 2\gamma n\rfloor \choose \lfloor 2\gamma n\rfloor}
& \le 2n\left(\frac{e(1+2\gamma)n}{2\gamma n}\right)^{2\gamma n} \nonumber\\
& =\exp\left(\log(2n)+n\cdot 2\gamma\left(1+\log\left(1+\frac{1}{2\gamma}\right)\right)\right) \label{connectedcount}
%& =\exp(o_{\gamma\rightarrow 0}(n)).
\end{align}
Since $x(1+\log(1/(2x)))\to 0$ as $x \downarrow 0$, we may choose $0 < \gamma < 9\beta c/(20R)$ small enough that (\ref{connectedcount}) is at most $\exp(n \cdot 9\beta c/(160R))$ for all $n$ sufficiently large.

Finally, for a fixed set $S$ with $9n/(10R)\le|S|\le(1-\beta)n$, the probability that $e_H(S, S^c)\le \gamma n$ is bounded above by $\p{Bin(\frac{9n}{10R}\cdot \beta n, p)\le \gamma n}$ since $e_H(S, S^c)$ stochastically dominates $Bin(|S|(n-|S|), p)$ and $|S|\ge \frac{9n}{10R}, n-|S|\ge \beta n$. Since $\frac{10R\gamma}{9\beta c}\le \frac{1}{2}$, by a Chernoff bound we have
\begin{equation}\label{lasttool}
\p{e_H(S, S^c)\le \gamma n} \le \p{Bin(\frac{9n}{10R}\cdot \beta n, p)\le \gamma n}\le \exp\left(-\frac{9\beta c}{80R} n\right).
\end{equation}
Since {\em all} sets $S \subset [n]$ with at least $\gamma n$ components (still viewed as subsets of the $n$-cycle) have $e(S,S^c) \ge \gamma n$, it follows by (\ref{connectedcount}), (\ref{lasttool}), and a union bound over sets with at most $\gamma n$ components that for all $n$ sufficiently large, %, there is $\gamma > 0$ such that for all $n$ large enough,
\begin{align}\label{wrapup3}
\p{\exists S  \in \bigcup_{9n/(10R) \leq j \leq (1-\beta)n} B_j(H), e_H(S,S^c) \leq \gamma n}
& \le 2n{n+\lfloor 2\gamma n\rfloor \choose \lfloor 2\gamma n\rfloor}
 \exp\left(-\frac{9\beta c}{80R} n\right) \nonumber\\
 & \le \exp\left(-\frac{9\beta c}{160R} n\right) \nonumber \\
& \leq \frac{1}{n^3}\, .
\end{align}
%where $\beta=\beta(c)$ is as in Corollary~\ref{cor:9over10}.
Writing $\alpha=\min(\gamma, \epsilon, \delta)$, it follows from
(\ref{wrapup1}), (\ref{wrapup2}), and (\ref{wrapup3}) that
\[
\p{\exists S  \in \bigcup_{R\log n \leq j \leq (1-\beta)n} B_j(H), e_H(S,S^c) \leq \alpha |S|} \leq \frac{3}{(n')^3}.
\]
Since $n' = n/R$ this completes the proof.
%Since, by Corollary~\ref{cor:9over10}, with probability at least $1-(1-\beta)^n$ all sets $S \subset [n]$ with $|S| \geq (1-\beta)n$
%have $e_H(S) \geq |E(H)|$, the remainder of the proof is exactly as in the case $c > x_k$.
\end{proof}
%\section{Conclusion}\label{sec:conc}

\vspace{-0.25cm}
\bibliographystyle{plainnat}

\appendix

\end{document}